\documentclass[11pt]{article}

\newtheorem{theorem}{Theorem}[section]
\newtheorem{lemma}[theorem]{Lemma}
\newtheorem{proposition}[theorem]{Proposition}
\newtheorem{corollary}[theorem]{Corollary}

\newenvironment{proof}[1][Proof]{\begin{trivlist}
		\item[\hskip \labelsep {\bfseries #1}]}{\end{trivlist}}
\newenvironment{definition}[1][Definition]{\begin{trivlist}
		\item[\hskip \labelsep {\bfseries #1}]}{\end{trivlist}}

\newenvironment{remark}[1][Remark]{\begin{trivlist}
		\item[\hskip \labelsep {\bfseries #1}]}{\end{trivlist}}

\newcommand{\qed}{\nobreak \ifvmode \relax \else
	\ifdim\lastskip<1.5em \hskip-\lastskip
	\hskip1.5em plus0em minus0.5em \fi \nobreak
	\vrule height0.75em width0.5em depth0.25em\fi}
\usepackage[margin=1in,includefoot]{geometry}

\usepackage{setspace}

\usepackage{tikz}
\usetikzlibrary{calc,patterns}
\usepackage{tikz-network}

\usepackage{cancel}

\usepackage{placeins}

\usepackage[utf8]{inputenc}

\usepackage{blindtext}
\usepackage{todonotes}

\usepackage{graphicx}
\usepackage{float}

\usepackage{xfrac}

\usepackage{fancyhdr}

\usepackage{mathtools}
\usepackage{nccmath}

\usepackage{relsize}

\usepackage{bm}

\usepackage{titling}
\pretitle{\begin{center}\Huge\bfseries}
	\posttitle{\par\end{center}\vskip 0.5em}
\preauthor{\begin{center}\Large\ttfamily}
	\postauthor{\end{center}}
\predate{\par\large\centering}
\postdate{\par}

\title{Determinantally equivalent nonzero functions}
\author{Harry Sapranidis Mantelos}
\date{\today}

\usepackage{bbm}

\usepackage{relsize}

\usepackage{mathtools}
\usepackage{csquotes}
\usepackage{mathrsfs}
\usepackage{amsbsy}

\usepackage{amssymb}

\usepackage{enumerate}
\usepackage{url}

\newcommand{\vertiii}[1]{{\left\vert\kern-0.25ex\left\vert\kern-0.25ex\left\vert #1 
		\right\vert\kern-0.25ex\right\vert\kern-0.25ex\right\vert}}

\allowdisplaybreaks
\begin{document}
	
	\maketitle
	\thispagestyle{empty}

	\begin{abstract}
		We study the problem raised in [Marco Stevens, Equivalent symmetric kernels of determinantal point processes, RMTA, 10(03):2150027, 2021] concerning the extension of its main result to the more general (potentially non-symmetric) setting. We construct a counterexample disproving the conjecture proposed in the paper, and subsequently solve it under some additional minor assumptions that preclude such counterexamples.
		
		The problem is plainly stated as follows: Let $\Lambda$ be a set and $\mathbb{F}$ a field, and suppose that $K,Q:\Lambda^2\to\mathbb{F}$ are two functions such that for any $n\in\mathbb{N}$ and $x_1,x_2,\ldots,x_n\in\Lambda$, the determinants of matrices $(K(x_i,x_j))_{1\leq i,j\leq n}$ and $(Q(x_i,x_j))_{1\leq i,j\leq n}$ agree. What are all the possible transformations that transform $Q$ into $K$? In [Marco Stevens, Equivalent symmetric kernels of determinantal point processes, RMTA, 10(03):2150027, 2021] the following two were conjectured: $(Tf)(x,y)=f(y,x)$; and $(Tf)(x,y)=g(x)g(y)^{-1}f(x,y)$ for some nowhere-zero function $g$. In the same paper, this conjectured classification is verified in the case of symmetric functions $K$ and $Q$. By extending the graph-theoretic techniques of the paper, we show that under some surprisingly simple and natural conditions the conjecture remains valid even with the symmetry constraints relaxed.
		
		By taking $\Lambda$ finite, the above problem, furthermore, reduces to that between two square matrices investigated in [Raphael Loewy,  Principal minors and diagonal similarity of matrices, Linear Algebra and its Applications 78 (1986), 23--64]. Hence, our paper presents a simple non-linear-algebraic proof that uses only some elementary combinatorics and  three simple algebraic identities involving $3$-cycles and $4$-cycles.
	\end{abstract}

	\section{Introduction and main result}
	In \cite{loewy1986principal}, Loewy investigates the relation between diagonal similarity and the concept of equal corresponding principle minors of two matrices. Diagonal similarity of matrices has long been studied in linear algebra; c.f., \cite{bassett1968qualitative, engel1973cyclic, fiedler1969cyclic, engel1980matrices}; as well as its connections with graph theory; cf., \cite{engel1982algorithms, saunders1978flows}. The allusion that it may have a connection with the concept of equal corresponding principle minors stems from the observation that if two $n\times n$ matrices $A$ and $B$ are diagonally similar (up to a transposition), i.e., $B=D^{-1}AD$ or $B=D^{-1}A^TD$ for some non-singular diagonal matrix $D$, then all their corresponding principle minors agree. It is then the converse statement and the extent of its validity that would naturally trigger one's curiosity. It was discovered by Loewy in the aforementioned paper that  the following two conditions (stated here informally) are enough for this converse to hold:
	\begin{itemize}
		\item $A$ is irreducible;
		\item some particular sub-matrices of $A$ have rank at least $2$.
	\end{itemize}
	Although purely linear-algebraic, the same exact problem can be found expressed, albeit in random matrix theory jargon, in probabilistic literature as well; and more specifically, in the theory of discrete determinantal point process (DPP). Said stochastic processes, in layman's terms, serve as models of random sets of finitely-many points. Their key feature is their ability to effectively model repulsion/diversity of points. This, and many other desirable features, has resulted in their ever-increasing popularity in the machine-learning community; cf. \cite{kulesza2012determinantal, tremblay2019determinantal, celis2018fair}. Their connection with Loewy's linear-algebraic problem stems from the fact that a DPP which models the configuration of, say $n$, random points which we label for simplicity $\{1,\ldots,n\}$, has a probability distribution that is completely characterized by a (deterministic) $n\times n$ matrix, say $K=(K_{ij})_{i,j=1}^n$, called the \textit{kernel} of the DPP, whereby the probability of a particular arrangement of points is given by the corresponding principle minor of the kernel; for example, the probability of observing the points $2$ and $3$ together is given by the principle minor $\begin{vmatrix} K_{22} & K_{23} \\ K_{32} & K_{33} \end{vmatrix}$. As such, it should be clear to the reader that the kernel of a (discrete) DPP is not unique. In particular, any $n\times n$ matrix $Q$ with equal corresponding principle minors (to those of $K$) is also a kernel of the DPP, and is commonly referred to in the literature as an \textit{equivalent kernel}. Thus, the formulation of Loewy's problem in this DPP language is as follows: given equivalent kernels $K$ and $Q$, to what extent is it true that they are diagonally similar? 
	
	Although (discrete) DPPs with non-symmetric kernels have been studied, e.g., \cite{gartrell2019learning, arnaud2024determinantal, borodin2010adding}; they have not received nearly the amount of attention their symmetric counterparts have, e.g., \cite{determinantal_thinning_of_pps_with_network, pmlr-v89-sadeghi19a, bardenet2015inference}. That being the case, it was of primary interest to answer the above question restricting to the case of symmetric (or Hermitian) equivalent kernels that is symmetric (or Hermitian) matrices with equal corresponding principle minors. It was discovered in \cite{kulesza2012learning} that in the (real or complex) symmetric setting, equivalent kernels $K$ and $Q$ will be diagonally similar with no additional conditions on the kernels/matrices. In \cite{launay2021determinantal}, the investigation was taken a bit further by restricting to the (more general) case of (complex) Hermitian matrices. The respective result follows from a specialization of \cite{loewy1986principal}.
	
	After this brief review of the origins and history of the original linear-algebraic problem, we now describe the functional extension we are concerned with. Let $K:\Lambda^2\to\mathbb{F}$  and $Q:\Lambda^2\to\mathbb{F}$ be two functions, where $\Lambda$ is some abstract set (of arbitrary cardinality) and $\mathbb{F}$ is an arbitrary field. Suppose that $K$ and $Q$ are \textit{determinantally equivalent} in the sense that
	\begin{equation}
		\label{equal_determinants_relation}
		\det(Q(x_i,x_j))_{i,j=1}^n=\det(K(x_i,x_j))_{i,j=1}^n\qquad \forall x_1,\ldots,x_n\in\Lambda\enspace\forall n\in\mathbb{N}.
	\end{equation}
	What are, then, all the possible transformations that transform $Q$ into $K$? In the paper that first formulated this problem, \cite{equiv_symm_kernels_for_dpps}, the following transformations were conjectured in Conjecture 1.4 therein:
	\begin{itemize}
		\item \textit{conjugation transformations}. There exists a nowhere-zero function $g:\Lambda\to\mathbb{F}$ such that for every $x,y\in\Lambda$,
		\begin{equation}
			\label{conjugation_transf}
			Q(x,y)=g(x)g(y)^{-1} K(x,y).
		\end{equation}
	We refer to $g$ as the \textit{conjugation function} of the transformation.
		\item \textit{transposition transformations}. $Q(x,y)=K(y,x)$ for every $x,y\in\Lambda$.
	\end{itemize}
	\begin{remark}
		Observe how these functional transformations are analogous to the concept of diagonal similarity between two matrices (up to a transposition). In fact, they are precisely just that when one works with $\Lambda$ finite.
	\end{remark}
	It was then proved in the same paper, via elementary techniques involving cycles, that restricting to symmetric functions $K$ and $Q$, the conjecture holds. It is also worth pointing out that this exact result, in the case when the underlying `set $\Lambda$' is finite, is precisely the relevant result from \cite{kulesza2012learning} involving symmetric matrices which we had discussed earlier. Both of the aforementioned results use very similar proof techniques. The problem of relaxing these symmetry constraints and solving this conjecture in that general setting was subsequently left open, and is what this paper is dedicated to. Specifically, we obtain the following result.
	\begin{theorem}
		\label{my_main_thm}
		Let $\Lambda$ be a set, $\mathbb{F}$ a field, and let $K,Q:\Lambda^2\to\mathbb{F}$ be two (not necessarily symmetric) nowhere-zero functions, except possibly on the set $\{(x,x):x\in\Lambda\}$. Suppose further that for every pairwise distinct $x,y,z,w\in\Lambda$,
		\begin{equation}
			\label{principle_minor_2_by_2_main_eqn}
			\begin{vmatrix}
				Q(x,y) & Q(x,w) \\
				Q(z,y) & Q(z,w)
			\end{vmatrix}\neq 0.
		\end{equation}
		If $K$ and $Q$ are determinantally equivalent, then $Q$ can be transformed into $K$ through only conjugation and transposition transformations. 
	\end{theorem}

We note that, in the particular case of $4<|\Lambda|<\infty$, condition (\ref{principle_minor_2_by_2_main_eqn}) on its own implies both of Loewy's conditions (from the two earlier bullet points): for the second bullet point, by looking up in \cite{loewy1986principal} the more precise formulation of the condition, one immediately sees this implication; for the first, suppose for a contradiction that $A$ is reducible and pick a row or column with three zeros (possible for $n\geq 5$); then at least two of these zeros with two other entries of $A$ show that (\ref{principle_minor_2_by_2_main_eqn}) does not hold. With that said, we should clarify that while our main result is related to a particular case of Loewy's result, the approach taken here is substantially different. In particular, our proof relies only on elementary arguments and avoids the more technical linear-algebraic machinery used by Loewy. This simplification is made possible by the stronger hypothesis assumed here regarding nowhere-zero functions, which is essential for applying the key tools developed in Section~\ref{section_graph_theory_tools}. Without this assumption, certain steps — specifically involving division by zero — would break down, and our method would not be viable. We believe that this elementary approach offers a more accessible perspective on the result and may be adaptable to other contexts.

	For completeness' sake, we give a few more words about DPPs without going too far out of focus from the paper's purpose. In particular, we note that besides discrete, there are also DPPs defined in the continuum, which have also been extensively studied; cf., \cite{hough2009zeros, soshnikov2000determinantal}. In that setting, there are, not a finite number of points which the DPP models, but uncountably-many. As such, it is clear that there can no longer be any square matrix of finite size that can characterize the DPP. Instead, it is now a measurable function, known as the \textit{correlation function}, that characterizes the DPP, which is of a special ``determinantal form". As one can imagine (and can indeed read from the above-cited texts) there are various additional measure-theoretic considerations that inevitably appear in the definition of such objects. For these very reasons it should be recognized that, in spite of the fact that it was of original interest, the problem Stevens' paper \cite{equiv_symm_kernels_for_dpps} introduces (and which we further investigate here) is independent of the theory of DPPs.
	
	The rest of the paper is structured as follows. We begin, in Section~\ref{counterexample_sec}, by discussing why the conjecture of \cite{equiv_symm_kernels_for_dpps} cannot hold in the general setting, and we then provide some insight into the hypotheses of our Theorem~\ref{my_main_thm}. In Section~\ref{sec_notation_and_terminology} we define and explain the various graph-theoretic objects we use to study the problem. In Section~\ref{section_short_pf} we provide a short proof of the main result of \cite{equiv_symm_kernels_for_dpps} with the additional assumption that the functions at hand, besides symmetric, are also nowhere-zero and the characteristic of the underlying field is not $2$. We do this by utilising ``a shortcut" given by Proposition~\ref{suff_conditions_for_cocycle_property}. Doing so gives us the opportunity to illustrate some of the techniques of said paper which we also deploy in proving our Theorem~\ref{my_main_thm}, and which were a source of inspiration for the conception of the novel techniques of the present paper. In Section~\ref{section_graph_theory_tools} we describe three simple algebraic identities using graphs, which our proof of the main result is essentially built upon. Lastly, in Section~\ref{section_proof_of_my_theorem}, by using the identities and graph-theoretic setting from Section~\ref{section_graph_theory_tools} we discover an underlying combinatorial structure in the problem; we derive various relations between determinantally equivalent functions, and we then use them to prove our theorem.
	
	\section{Counterexamples to the conjecture} \label{counterexample_sec}
	
	For reasons that will be explained later on in the section, our analysis is restricted to functions $K,Q:\Lambda\to\mathbb{F}$ that are nowhere-zero except possibly on the set $\{(x,x):x\in \Lambda\}$. A simple and canonical counterexample to the conjecture of \cite{equiv_symm_kernels_for_dpps} is the following: Let $\Lambda$ be a set such that $|\Lambda|=4$ - and so then the problem is actually none other than the linear-algebraic one between two matrices described in the introductory section of the paper. For simplicity, let us assume $\Lambda=\{1,2,3,4\}$. Consider the functions/matrices $K:\Lambda^2\to\mathbb{F}$ and $Q:\Lambda^2\to\mathbb{F}$ defined by
	
	\begin{equation}
		\label{example_H_J_matrices}
		(K(x,y))_{1\leq x,y\leq 4}=\begin{pmatrix}
			a & b & 1 & 1 \\
			c & d & 1 & 1 \\
			1 & 1 & e & f \\
			1 & 1 & g & h
		\end{pmatrix}, \qquad (Q(x,y))_{1\leq x,y\leq 4}=\begin{pmatrix}
			a & c & 1 & 1 \\
			b & d & 1 & 1 \\
			1 & 1 & e & f \\
			1 & 1 & g & h
		\end{pmatrix},
	\end{equation}
for some $b, c, f, g\in\mathbb{F}\setminus{\{0\}}$ and $a, d, e, h\in\mathbb{F}$ such that $c\neq b$ and $f\neq g$. In this case, one finds, through a straightforward application of the formula for determinants of block matrices, that $K$ is indeed determinantally equivalent to $Q$. However, one will also find that $K$ cannot be obtained from $Q$ by applying transposition and conjugation transformations to $Q$. Thus, the conjectured classification of transformations does not hold: we have just identified one additional type of transformation from the two conjectured which yields determinantally equivalent functions; perhaps a suitable name for it would be a ``\textit{partial transposition transformation}", since only \textit{part} of the matrix is being transposed. Said in linear-algebraic language, while the above two matrices have equal corresponding principle minors, they are not diagonally similar (up to a transposition), but only ``partially".
	
	 Our first step is to seek for a condition to impose on an arbitrary pair of determinantally equivalent functions that rules out the pair $K$ and $Q$ of the previously-described block forms. The sub-matrices of ones in the upper-right and lower-left blocks of the above two block matrices are the ones that bring about the block structure that enables this ``partial transposition transformation" discussed previously to be a viable transformation that transforms the two matrices with equal corresponding principle minors into one another. And so by far the most intuitive condition to impose on our pair of matrices that we hope would be enough to make this type of transformation non-permissible would be to require non-zero determinants in the upper-right and lower-left $2\times 2$ sub-matrix regions, that is, to require $$\begin{vmatrix} Q(1,3) & Q(1,4) \\ Q(2,3) & Q(2,4) \end{vmatrix}\neq 0,\quad \begin{vmatrix} Q(3,1) & Q(3,2) \\ Q(4,1) & Q(4,2) \end{vmatrix} \neq 0,$$
	which, in our current case of $|\Lambda|=4$, is, up to simultaneous permutations of rows and columns, the same as requiring for every pairwise distinct $x,y,z,w\in\Lambda$,
	\begin{equation*}
		\begin{vmatrix}
			Q(x,y) & Q(x,w) \\
			Q(z,y) & Q(z,w)
		\end{vmatrix}\neq 0.
	\end{equation*}
 In our paper, we show that this condition, in conjunction with the nowhere-zero condition stated earlier, in fact, are sufficient in solving the conjecture for a general abstract set $\Lambda$.

It remains to explain the reason for our `nowhere-zero functions' requirement. Plainly, it all comes down to the fact that two of the three algebraic identities involving $3$-cycles and $4$-cycles (cf., Section~\ref{section_graph_theory_tools}), which our proof of Theorem~\ref{my_main_thm} in Section~\ref{section_proof_of_my_theorem} is essentially built upon, involve divisors that are of the form $h(x,y)$, where $x\neq y$, with $h=K$ or $h=Q$; in which case, the condition $h(x,y)\neq 0$ is, of course, necessary.

After the statement of Theorem~\ref{my_main_thm} we had remarked that in the case of $4<|\Lambda|<\infty$, condition (\ref{principle_minor_2_by_2_main_eqn}) on its own implies both of Loewy's sufficient conditions from \cite{loewy1986principal}. So, although the proof we give in the paper certainly requires the `nowhere-zero functions' condition, it remains an open question whether the proof can be adapted to require weaker assumptions.
	
	\section{Terminology and notations}
	\label{sec_notation_and_terminology}
	For the most part, we will be abiding to standard graph-theoretic notations and terminology. Of course, some terminology in graph theory will always inevitably vary among authors. For this reason, to avoid confusion, we take the time in this section to carefully go through and explain the technical language we employ in the forthcoming sections, most of which will evade many of the fine and highly-specific technicalities that are commonly found in graph-theoretic literature, as we will simply not be needing them for this paper.
	
\begin{definition}
Let $\mathcal{M}$ be a (possibly infinite) set. For an integer $1\leq n\leq |\mathcal{M}|$, we define a \textit{cycle of length $n$ in $\mathcal{M}$} (or an \textit{$n$-cycle in $\mathcal{M}$}, for short) to be an $(n+1)$-tuple of the form $p=(p_0,p_1,\ldots,p_n)\in\mathcal{M}^{n+1}$, where each of the $p_i$ are distinct except for $p_0=p_n$. We will refer to the $p_i$ as the \textit{vertices of the cycle $p$}.
\end{definition}

\begin{definition}
Let $\mathcal{M}$ be a (possibly infinite) set and fix an integer $1\leq n\leq |\mathcal{M}|$. Given an $n$-cycle $p=(p_0,p_1,\ldots,p_n)$ in $\mathcal{M}$, we denote by $p'\coloneqq (p_n,p_{n-1},\ldots,p_0)$ the \textit{cycle $p$ in reverse}.
\end{definition}

\begin{definition}
Let $\mathcal{M}$ be a (possibly infinite) set and fix an integer $1\leq n\leq |\mathcal{M}|$. In this paper, we will occasionally speak of \textit{undirected $n$-cycles in $\mathcal{M}$}; by this we will mean $n$-cycles in $\mathcal{M}$ \textit{up to reversion}. 
\end{definition}

As a simple example of the previous definition, consider the case when $\mathcal{M}=\{1,2,3,4\}$: there are exactly \textit{six} $4$-cycles in $\mathcal{M}$, namely $(1,2,3,4,1)$, $(1,4,3,2,1)$, $(1,2,4,3,1)$, $(1,3,4,2,1)$, $(1,3,2,4,1)$, $(1,4,2,3)$; but there are only \textit{three} undirected $4$-cycles in $\mathcal{M}$. The below figure makes this clear.

\begin{tikzpicture}[scale=0.9]
	\Vertex[label=$1$]{A} \Vertex[x=5,label=$4$]{D} \Vertex[x=1,y=2, label=$2$]{B} \Vertex[x=4,y=2, label=$3$]{C}
	\Edge(A)(B)
	\Edge(B)(C)
	\Edge(C)(D)
	\Edge(D)(A)
\end{tikzpicture}
\hspace{1pt}
\begin{tikzpicture}[scale=0.9]
	\Vertex[label=$1$]{A} \Vertex[x=5,label=$4$]{D} \Vertex[x=1,y=2, label=$2$]{B} \Vertex[x=4,y=2, label=$3$]{C}
	\Edge(A)(C)
	\Edge(D)(B)
	\Edge(B)(A)
	\Edge(C)(D)
\end{tikzpicture}
\hspace{1pt}
\begin{tikzpicture}[scale=0.9]
	\Vertex[label=$1$]{A} \Vertex[x=5,label=$4$]{D} \Vertex[x=1,y=2, label=$2$]{B} \Vertex[x=4,y=2, label=$3$]{C}
	\Edge(A)(D)
	\Edge(D)(B)
	\Edge(B)(C)
	\Edge(A)(C)
\end{tikzpicture}
 
	\begin{definition}
	Let $\Lambda$ be a (possibly infinite) set and let $\mathbb{F}$ be a field. Fix an integer $1\leq n\leq|\Lambda|$. For a two-variable function $h:\Lambda^2\to\mathbb{F}$ and an $n$-cycle, $p=(p_0,\ldots,p_n)$, in $\Lambda$, we denote by $h[p]$ the product
		\begin{equation*}
			h[p]\coloneqq\prod_{i=1}^n h(p_{i-1},p_i);
		\end{equation*}
		and by $h'[p]$ the analogous product with respect to the reverse cycle $p'$:
		\begin{equation*}
			h'[p]\coloneqq\prod_{i=1}^n h(p_i,p_{i-1}) \enspace(=h[p']).
		\end{equation*}
	\end{definition}
	
	\begin{remark}
		Let us define the function $\tau:\{0,1,\ldots,n\}\to\{0,1,\ldots,n-1\}$ to be the following shift operator:
		\begin{equation*}
			\tau(k)\equiv k+1 \pmod{n},\qquad k\in\{0,1\ldots,n\}.
		\end{equation*}
		Though this is a trivial clarification, it is worth pointing out that for an $n$-cycle $p=(p_i)_{i=0}^n$ and any $j\in\mathbb{N}$, the $n$-cycle $(p_{\tau^j(i)})_{i=0}^n$, where $\tau^j$ denotes the $j$-th iteration of the shift operator $\tau$, is indistinguishable from $p$ in the sense that it describes the exact same graph-theoretic object; and obviously we have $h[p]=h[(p_{\tau^j(i)})_{i=0}^n]$ for any two-variable function $h$. 
		
	\end{remark}
	
	\begin{definition}
		Let $\Lambda$ be a (possibly infinite) set and let $\mathbb{F}$ be a field. Fix an integer $1\leq n\leq|\Lambda|$. We say that a two-variable function $c:\Lambda^2\to\mathbb{F}$ satisfies the \textit{cocycle property for $n$-cycles} if for every $n$-cycle $p$ in $\Lambda$, $c[p]=1$. 
		
		If $c$ satisfies the cocycle property for cycles in $\Lambda$ of every length $n$, where $1\leq n \leq |\Lambda|$, we say that $c$ is a (full) \textit{cocycle function}. Equivalently, $c$ is a cocycle function if for all $z,w\in\Lambda$, $c(z,z)=1$ and $c(z,w)c(w,z)=1$, and for all integers $r>2$ and every $r$-tuple $(z_1,\ldots,z_r)\in\Lambda^r$,
		\begin{equation}
			\label{cocycle_property}
			c(z_1,z_2)c(z_2,z_3)\cdots c(z_{r-1},z_r)c(z_r,z_1)=1.
		\end{equation}
	\end{definition}

	\section{Proof of the symmetric case} \label{section_short_pf}
	In addition to our own, in later sections, we have employed and taken inspiration from some of the tools and techniques from \cite{equiv_symm_kernels_for_dpps}. Thus, we believe it is worth starting off by giving a short proof of said paper's main result under the further assumption that the functions at hand are nowhere-zero (except perhaps on the set  $\{(x,x):x\in\Lambda\}$). In this way, we hope the reader gets a clear illustration of the main concepts in the simpler symmetric setting before we move on to extend them to the more involved general setting. We emphasize that the techniques in this section (in particular, the proof of Proposition~\ref{suff_conditions_for_cocycle_property}) are fairly standard. 
	
	We assume throughout this section that the characteristic of the field $\mathbb{F}$ is \textit{not} two. The result is stated as follows.
	
	\begin{theorem}[Modified Theorem 1.5 from \cite{equiv_symm_kernels_for_dpps}]
		\label{thm_from_symm_equiv_kernels}
		Let $\Lambda$ be a set, $\mathbb{F}$ a field, and let $K,Q:\Lambda^2\to\mathbb{F}$ be symmetric functions that are nowhere-zero except perhaps on the set $\{(x,x):x\in\Lambda\}$. If $K$ and $Q$ are determinantally equivalent (i.e., equation (\ref{equal_determinants_relation}) holds), then it must be the case that $K$ and $Q$ are conjugation transformations of one another.
	\end{theorem}
	
	\begin{definition}
		If $Q:\Lambda^2\to\mathbb{F}$ and $K:\Lambda^2\to\mathbb{F}$ are functions that satisfy
		\begin{equation*}
			Q(x,y)=c(x,y)K(x,y)\qquad\forall x,y\in\Lambda
		\end{equation*}
		for some cocycle function $c:\Lambda^2\to\mathbb{F}$, we say that $Q$ is a \textit{cocycle transformation} of $K$.
	\end{definition}
	As remarked in \cite{equiv_symm_kernels_for_dpps}, we make the following key observation.
	\begin{proposition}
		\label{conj_transf_cocycle_transf_equivalence}
		Let $\Lambda$ be a set, $\mathbb{F}$ a field, and let $K:\Lambda^2\to\mathbb{F}$ and $Q:\Lambda^2\to\mathbb{F}$ be two (not necessarily nowhere-zero) functions of two variables. Then, $Q$ is a conjugation transformation of $K$ if and only if $Q$ is a cocycle transformation of $K$.
	\end{proposition}
	\begin{proof}
		Suppose there exists a nowhere-zero function $g:\Lambda\to\mathbb{F}$ such that $$Q(x,y)=g(x)g(y)^{-1}K(x,y)\quad\text{ $\forall x,y\in\Lambda$.}$$ Consider the function $c:\Lambda^2\to\mathbb{F}$, given by $$c(x,y)=g(x)g(y)^{-1},\qquad x,y\in\Lambda.$$ It is then not difficult to see that $c$ defined in this way is a cocycle function, and hence, that $Q$ is a cocycle transformation of $K$ with respective cocycle function $c$.
		
		Conversely, suppose there exists a cocycle function $c:\Lambda^2\to\mathbb{F}$ such that $$Q(x,y)=c(x,y)K(x,y)\quad\enspace\forall x,y\in\Lambda.$$ Let $x,y\in\Lambda$ and fix an arbitrary $x_0\in\Lambda$. Thanks to the cocycle property, $$c(x,y)c(y,x_0)c(x_0,x)=1,$$ and thus $$c(x,y)=\frac{c(x_0,x)^{-1}}{c(y,x_0)}.$$ But also, again thanks to the cocycle property, we have $$c(x,x_0)=c(x_0,x)^{-1}.$$ Therefore, $$c(x,y)=\frac{c(x,x_0)}{c(y,x_0)}.$$ Consider the nowhere-zero function $g:\Lambda\to\mathbb{F}$, given by $$g(z)=c(z,x_0),\qquad z\in\Lambda.$$
		It is then clear that $Q$ is a conjugation transformation of $K$ with conjugation function $g$. \qed 
	\end{proof}
	
	We now state and prove the ``shortcut result" we had mentioned earlier in the introduction.
	\begin{proposition}
		\label{suff_conditions_for_cocycle_property}
		Let $\Lambda$ be a set and $\mathbb{F}$ a field. If a two-variable function $c:\Lambda^2\to\mathbb{F}$ satisfies
		\begin{enumerate}[\itshape(i)]
			\item{$c(x,x)=1$ for every $x\in\Lambda$;} \label{property_i_of_suff_cocycle}
			\item{$c(x,y)c(y,x)=1$ for every $x,y\in\Lambda$;} \label{property_ii_of_suff_cocycle}
			\item{$c(x,y)c(y,z)c(z,x)=1$ for every $x,y,z\in\Lambda$,} \label{property_iii_of_suff_cocycle}
		\end{enumerate}
		then $c$ is a cocycle function.
		
		In other words, $c$ satisfying the cocycle property for cycles of lengths $1$, $2$ and $3$ is both a necessary and sufficient condition for $c$ to be a (full) cocycle function, that is, for it to satisfy the cocycle property for cycles of any length.
	\end{proposition}
	\begin{proof}
	 By (\ref{property_i_of_suff_cocycle}) and (\ref{property_ii_of_suff_cocycle}) of the hypothesis, we already have for all $z,w\in\Lambda$, $c(z,z)=1$ and $c(z,w)c(w,z)=1$. We proceed by induction on $r\geq3$ in (\ref{cocycle_property}). The base case $r=3$ is already satisfied by $c$ thanks to (\ref{property_iii_of_suff_cocycle}) of the hypothesis. We now move on to the inductive step. To this end, let $(z_1,\ldots,z_{r+1})\in\Lambda^{r+1}$ be an $(r+1)$-tuple. Then,
		\begin{align*}
			&c(z_1,z_2)c(z_2,z_3)\cdots c(z_{r},z_{r+1})c(z_{r+1},z_1) \\ &=c(z_1,z_2)c(z_2,z_3)c(z_3,z_1) c(z_{3},z_1)^{-1} \Big( \prod_{i=4}^{r+1} c(z_{i-1},z_i) \Big) c(z_{r+1},z_1)\\
			&=c(z_1,z_3)\Big( \prod_{i=4}^{r+1} c(z_{i-1},z_i) \Big) c(z_{r+1},z_1)&\text{(by (\ref{property_iii_of_suff_cocycle}) and (\ref{property_ii_of_suff_cocycle}))} \\
			&=1,
		\end{align*}
		where the last equality is due to the inductive hypothesis, since $(z_1,z_3,z_4,\ldots,z_{r+1})\in\Lambda^{r}$ is an $r$-tuple. \qed
	\end{proof}
	
	We are now in the position to provide a short proof of the (modified) main result of \cite{equiv_symm_kernels_for_dpps}:
	\begin{proof}[Proof of Theorem~\ref{thm_from_symm_equiv_kernels}]
		Consider the function $S:\Lambda^2\to\mathbb{F}$, given by $$S(x,y)=\begin{cases*} \frac{Q(x,y)}{K(x,y)}, &\text{if $x\neq y$} \\
			1, &\text{if $x=y$}\end{cases*},\qquad x,y\in\Lambda.$$ Since $Q(x,y)=S(x,y)K(x,y)$ for every $x,y\in\Lambda$, if we can show that $S$ is a cocycle function, the result will follow immediately from Proposition~\ref{conj_transf_cocycle_transf_equivalence}.
		
		By equation (\ref{equal_determinants_relation}) with $n\in\{1,2\}$, $S$ satisfies conditions (\ref{property_i_of_suff_cocycle}) and (\ref{property_ii_of_suff_cocycle}) of Proposition~\ref{suff_conditions_for_cocycle_property}. To see that property (\ref{property_iii_of_suff_cocycle}) of the proposition is also satisfied, we make use of equation (\ref{equal_determinants_relation}) with $n=3$ and the Leibniz formula for determinants: let $x_1,x_2,x_3\in\Lambda$ be pairwise distinct, then
		\begin{equation}
			\label{leibniz_equality_for_3_by_3}
			\sum_{\sigma\in S_3} \text{sgn}(\sigma) \prod_{i=1}^3 K(x_i,x_{\sigma(i)})=\sum_{\sigma\in S_3} \text{sgn}(\sigma) \prod_{i=1}^3 Q(x_i,x_{\sigma(i)}).
		\end{equation}
	If a permutation $\sigma\in S_3$ fixes a point, then since $S$ satisfies properties (i) and (ii) of Proposition~\ref{suff_conditions_for_cocycle_property} we have
		\begin{equation}
			\label{eq}
			\prod_{i=1}^3 K(x_i,x_{\sigma(i)}) = \prod_{i=1}^3 Q(x_i,x_{\sigma(i)}).
		\end{equation}
 We can then subtract these terms from (\ref{leibniz_equality_for_3_by_3}) and be left with the terms that come from the permutations $\sigma=(1 2 3)$ and $\sigma'=(3 2 1)$; in which case, (\ref{leibniz_equality_for_3_by_3}) simplifies to
		\begin{equation*}
			K(x_1,x_2)K(x_2,x_3)K(x_3,x_1)+K(x_3,x_2)K(x_2,x_1)K(x_1,x_3)
		\end{equation*}
		\begin{equation}
			\label{leibniz_3_by_3_equality}
			=
		\end{equation}
		\begin{equation*}
			Q(x_1,x_2)Q(x_2,x_3)Q(x_3,x_1)+Q(x_3,x_2)Q(x_2,x_1)Q(x_1,x_3),
		\end{equation*}
		which, by the symmetry assumptions on both $K$ and $Q$, further simplifies to
		\begin{equation}
			\label{leibniz_3_by_3_equality_symm}
			2K(x_1,x_2)K(x_2,x_3)K(x_3,x_1)=2Q(x_1,x_2)Q(x_2,x_3)Q(x_3,x_1).
		\end{equation}
		Hence, property (\ref{property_iii_of_suff_cocycle}) of Proposition~\ref{suff_conditions_for_cocycle_property} is also satisfied by $S$. \qed
	\end{proof}

	\section{Cycles on $4$ vertices} \label{section_graph_theory_tools}

As alluded to in Section~\ref{section_short_pf}, cycles will also play a key role in deriving the main result of this paper. In essence, as we shall see in the final section, the proof of our Theorem~\ref{my_main_thm} comes down to establishing the cocycle property for cycles of length $3$ in $\Lambda$ of one of two particular functions of two variables. Foreshadowing the proof of Theorem~\ref{my_main_thm} some more, in addition to equation (\ref{equal_determinants_relation}) with $n=3$, we will also be making use of equation (\ref{equal_determinants_relation}) with $n=4$. This, along with the usual application of the Leibniz formula for determinants, will mean that cycles of length $4$ will also enter the scene. Finding suitable ways to link such cycles with ones that are of lengths $2$ and $3$ will allow us to make use of various identities already in our possession from the previous section regarding determinantally equivalent functions, as well as some additional ones we establish later on. This is purely a graph-theoretic matter and is what the remainder of this section is concerned with. 
	
	For the remainder of this section, we set $\mathcal{M}\coloneqq\{1,2,3,4\}$. We had explained in Section~\ref{sec_notation_and_terminology} that there are exactly six distinct directed $4$-cycles in the set $\mathcal{M}$. It is not difficult to list them, for there are only three distinct undirected $4$-cycles in $\mathcal{M}$, and each has exactly two possible directions for its trail.

	We will use the following labelling for the $4$-cycles in $\mathcal{M}$ henceforth:
\begin{equation}
	\label{four_cycle_labellings}
q^{[1]}\coloneqq (1,2,3,4,1),\qquad q^{[2]}\coloneqq (1,2,4,3,1),\qquad q^{[3]}\coloneqq (1,3,2,4,1).
\end{equation}
	In the same way, we also know that there are exactly eight distinct directed $3$-cycles in $\mathcal{M}$; we will use the following labelling henceforth:
	\begin{equation}
		\label{three_cycle_labellings}
		p^{(1)}\coloneqq (1,2,3,1),\qquad p^{(2)}\coloneqq (1,2,4,1),\qquad p^{(3)}\coloneqq (1,3,4,1),\qquad p^{(4)}\coloneqq (2,3,4,2).
	\end{equation}

	\begin{lemma}
		\label{graph_theory_lemma_3_from_notes}
		Let $\mathbb{F}$ be a field, and let $h:\mathcal{M}^2\to\mathbb{F}$ be a function of two variables. Then,
		\begin{equation}
			\label{graph_theory_lemma_3_from_notes_eq}
			h[q^{[2]}]h[q^{[3]}]=h(1,3)h(3,1)\cdot h[p^{(2)}]h'[p^{(4)}];
		\end{equation}
	and
		\begin{equation}
			\label{graph_theory_lemma_3_from_notes_eq2}
			h[q^{[2]}]h'[q^{[3]}]=h(2,4)h(4,2)\cdot h[p^{(1)}]h'[p^{(3)}].
		\end{equation}
	\end{lemma}
	\begin{proof}
It is not difficult to verify the two equations. \qed
	\end{proof}

	\begin{lemma}
		\label{graph_theory_lemma_1_from_notes}
	Let $\mathbb{F}$ be a field, and let $h:\mathcal{M}^2\to\mathbb{F}$ be a function of two variables that is nowhere-zero except possibly on the set $\{(x,x):x\in\mathcal{M}\}$. Then,
		\begin{equation}
			\label{graph_theory_lemma_1_from_notes_eq1}
			h[q^{[1]}]=\frac{h[p^{(1)}]h[p^{(3)}]}{h(1,3)h(3,1)}=\frac{h[p^{(2)}]h[p^{(4)}]}{h(2,4)h(4,2)}.
		\end{equation}
	\end{lemma}
	\begin{proof}
		In the below figure we provide a pictorial proof of the lemma:
		\begin{figure}[H]
			\centering
			\begin{tikzpicture}
				\Vertex[label=$1$]{A} \Vertex[x=5,label=$4$]{D} \Vertex[x=1,y=2, label=$2$]{B} \Vertex[x=4,y=2, label=$3$]{C}
				\Edge[Direct,color=red](A)(B)
				\Edge[Direct,color=red](B)(C)
				\Edge[Direct,color=blue](C)(D)
				\Edge[Direct,color=blue](D)(A)
				\Edge[Direct,bend=10,color=blue](A)(C)
				\Edge[Direct,bend=10,color=red](C)(A)	
			\end{tikzpicture}
			\hfill
			\begin{tikzpicture}
				\Vertex[label=$1$]{A} \Vertex[x=5,label=$4$]{D} \Vertex[x=1,y=2, label=$2$]{B} \Vertex[x=4,y=2, label=$3$]{C}
				\Edge[Direct,color=red](A)(B)
				\Edge[Direct,color=blue](B)(C)
				\Edge[Direct,color=blue](C)(D)
				\Edge[Direct,color=red](D)(A)
				\Edge[Direct,bend=10,color=red](B)(D)
				\Edge[Direct,bend=10,color=blue](D)(B)		
			\end{tikzpicture} 
			\caption{In the graph on the left, in red is the $3$-cycle $\color{red}p^{(1)}$, and in blue is the $3$-cycle $\color{blue} p^{(3)}$. In the graph on the right, in red is the $3$-cycle $\color{red}p^{(2)}$, and in blue is the $3$-cycle $\color{blue} p^{(4)}$. \qed} 
		\end{figure} 
	\end{proof}

	\begin{lemma}
		\label{graph_theory_lemma_2_from_notes}
		Let $\mathbb{F}$ be a field, and let $h:\mathcal{M}^2\to\mathbb{F}$ be a function of two variables that is nowhere-zero except possibly on the set $\{(x,x):x\in\mathcal{M}\}$. Then,
		\begin{equation*}
			h[p^{(1)}]=\frac{h[p^{(2)}]h'[p^{(3)}]h[p^{(4)}]}{h(4,1)h(1,4)\cdot h(3,4)h(4,3)\cdot h(2,4)h(4,2)}.
		\end{equation*}
	\end{lemma}
	\begin{proof}
		In the below figure we provide a pictorial proof of the lemma.
		\begin{figure}[H]
			\centering
			\begin{tikzpicture}
				\Vertex[label=$3$]{A} \Vertex[x=6,label=$2$]{B}, \Vertex[x=3,y=3,label=$1$]{C} \Vertex[x=3,y=1,label=$4$]{D}
				\Edge[Direct,color=blue](A)(C)
				\Edge[Direct,color=red](C)(B)
				\Edge[Direct,color=green](B)(A)
				\Edge[Direct,color=blue,bend=20](C)(D)
				\Edge[Direct,color=red,bend=20](D)(C)
				\Edge[Direct,color=green,bend=10](A)(D)
				\Edge[Direct,color=blue,bend=10](D)(A)
				\Edge[Direct,color=red,bend=10](B)(D)
				\Edge[Direct,color=green,bend=10](D)(B)
			\end{tikzpicture}
			\caption{In red is the $3$-cycle $\color{red} p^{(2)}$, in green is the $3$-cycle $\color{green} p^{(4)}$, and in blue is the $3$-cycle $\color{blue} p^{(3)}$ in reverse. \qed}
		\end{figure}
	\end{proof}

	\section{Proof of the main result} 
	\label{section_proof_of_my_theorem}
	The backbone of the proof of our main theorem is the following elementary fact.
	\begin{lemma}
		\label{key_trick_lemma}
		Let $a,b,a',b'\in\mathbb{F}$ satisfy
		\begin{equation}
			\label{key_trick_lemma_eq1}
			a+b=a'+b'
		\end{equation}
		and
		\begin{equation}
			\label{key_trick_lemma_eq2}
			ab=a'b'.
		\end{equation}
		Then, we either have $(i)$ $a=a'$ and $b=b'$; or $(ii)$ $a=b'$ and $b=a'$.
	\end{lemma}
	\begin{proof}
		Consider the two quadratic equations $$x^2-(a+b)x+ab=0$$ and $$x^2-(a'+b')x+a'b'=0.$$
		The first equation has roots $a$ and $b$, and the second equation has roots $a'$ and $b'$. Thanks to equations (\ref{key_trick_lemma_eq1}) and (\ref{key_trick_lemma_eq2}) the two quadratic equations are equal, and so must be the roots. \qed
	\end{proof}

For the same reasons as in Section~\ref{section_short_pf}, to prove Theorem~\ref{my_main_thm} it will suffice to prove that, for the pair of determinantally equivalent functions $K$ and $Q$ from the statement, either the function $S:\Lambda^2\to\mathbb{F},$ given by $$S(x,y)=\begin{cases*}
	\frac{Q(x,y)}{K(x,y)}, &\text{if $x\neq y$} \\
	1, &\text{if $x=y$}
\end{cases*},\qquad x,y\in\Lambda,$$ or the function $\tilde{S}:\Lambda^2\to\mathbb{F},$ given by $$\tilde{S}(x,y)=\begin{cases*}
	\frac{Q(x,y)}{K(y,x)}, &\text{if $x\neq y$} \\
	1, &\text{if $x=y$}
\end{cases*},\qquad x,y\in\Lambda,$$ is a cocycle function.

Utilizing equation (\ref{equal_determinants_relation}) with $n\in\{1,2\}$, it is easy to check that both $S$ and $\tilde{S}$ satisfy the cocycle property for cycles in $\Lambda$ of lengths $1$ and $2$, and that $Q(x,y)=S(x,y)K(x,y)=\tilde{S}(x,y)K(y,x)$ for every $x,y\in\Lambda$. Indeed, for every $x,y\in\Lambda$ distinct, equation (\ref{equal_determinants_relation}) with $n\in\{1,2\}$ yields
\begin{equation*}
K(x,x)=Q(x,x)
\end{equation*}
and
\begin{equation}
\label{eqn_2by2_1by1}
K(x,y)K(y,x)=Q(x,y)Q(y,x).
\end{equation}
Thus, by Proposition~\ref{suff_conditions_for_cocycle_property} it suffices to show that either $S$ or $\tilde{S}$ satisfies the cocycle property for cycles of length $3$ in $\Lambda$. Expressing this in our shorthand notation, we need to prove that either
\begin{equation}
	\label{either_S_is_cocycle}
	S[p]=1\enspace\text{for every $3$-cycle $p\coloneqq (p_i)_{i=0}^3$ in $\Lambda$;}
\end{equation}
or
\begin{equation}
	\label{either_tilde_S_is_cocycle}
	\tilde{S}[p]=1\enspace\text{for every $3$-cycle $p\coloneqq (p_i)_{i=0}^3$ in $\Lambda$.}
\end{equation}

\begin{definition}
Let $\Lambda$ be a set, $\mathbb{F}$ a field, and let $K,Q:\Lambda^2\to\mathbb{F}$ be two functions of two-variables. We say that a $3$-cycle, $p=(p_i)_{i=0}^3$, in $\Lambda$ belongs to Case 1 or Case 2 if
\begin{equation*}
	\text{\textbf{Case 1:} } K[p]=Q[p] \text{ and } K'[p]=Q'[p];
\end{equation*}
or
\begin{equation*}
	\text{\textbf{Case 2:} } K[p]=Q'[p] \text{ and } K'[p]=Q[p],
\end{equation*}
respectively.
\end{definition}

\begin{remark}
Notice how if $K$ and $Q$ in the above definition are determinantally equivalent and nowhere-zero except possibly on the set $\{(x,x):x\in\Lambda\}$, then in both Case 1 and Case 2 the first equality implies the second, and vice versa. For example, if $K[p]=Q[p]$, then
\begin{equation*}
K'[p]=\frac{K[p]\cdot K'[p]}{K[p]}=\frac{Q[p]\cdot Q'[p]}{Q[p]}=Q'[p],
\end{equation*}
where the second equality is an immediate consequence of (\ref{eqn_2by2_1by1}). The remaining implications can be verified using the same argument. This fact is implicitly used in much of the remaining paper.
\end{remark}

We make an important observation:
\begin{lemma}
	\label{prop_about_case1_case2_for_3_cycles}
	Let $\Lambda$ be a set, and $\mathbb{F}$ a field. If functions $K:\Lambda^2\to\mathbb{F}$ and $Q:\Lambda^2\to\mathbb{F}$ are determinantally equivalent, then every $3$-cycle, $p=(p_i)_{i=0}^3$, in $\Lambda$ belongs to at least one of Case 1 or Case 2.
\end{lemma}
\begin{proof}
	As we had seen in Section~\ref{section_short_pf} (recall equation  (\ref{leibniz_equality_for_3_by_3})), equation (\ref{equal_determinants_relation}) with $n=3$ in conjunction with the Leibniz formula for determinants yields
	\begin{equation*}
		\sum_{\sigma\in S_3} \text{sgn}(\sigma) \prod_{i=1}^3 K(p_i,p_{\sigma(i)}) = \sum_{\sigma\in S_3} \text{sgn}(\sigma) \prod_{i=1}^3 Q(p_i,p_{\sigma(i)}).
	\end{equation*}
	By recalling the arguments from the proof of Theorem~\ref{thm_from_symm_equiv_kernels}, we can conclude that
	\begin{equation}
		\label{key_trick_mythm_eq1}
		K[p]+K'[p]=Q[p]+Q'[p].
	\end{equation}
	But, thanks to equation (\ref{eqn_2by2_1by1}), we also have
	\begin{equation}
		\label{key_trick_mythm_eq2}
		K[p]K'[p]=Q[p]Q'[p].
	\end{equation}
	Equations (\ref{key_trick_mythm_eq1}) and (\ref{key_trick_mythm_eq2}) then allow us to apply Lemma~\ref{key_trick_lemma} to complete the proof. \qed
\end{proof}

\begin{remark}
	The conclusion of the above lemma holds true even when the determinantally equivalent functions $K$ and $Q$ are not necessarily nowhere-zero.
\end{remark}
	
	Notice now how, thanks to Lemma~\ref{prop_about_case1_case2_for_3_cycles}, proving that either (\ref{either_S_is_cocycle}) or (\ref{either_tilde_S_is_cocycle}) is satisfied is equivalent to showing that $3$-cycles in $\Lambda$ belong to the same Case. It turns out that it suffices to prove this statement for every subset $\mathcal{M}\subseteq \Lambda$ with $|\mathcal{M}|=4$. More specifically, it suffices to prove the following proposition:
	
	\begin{proposition}
		\label{prop_step_2_lambda_4_elements}
		Let $\Lambda$ be a set and $\mathbb{F}$ be a field. Suppose that functions $K:\Lambda^2\to\mathbb{F}$ and $Q:\Lambda^2\to\mathbb{F}$ are determinantally equivalent and nowhere-zero except possibly on the set $\{(x,x):x\in\Lambda\}$. Suppose further that condition (\ref{principle_minor_2_by_2_main_eqn}) is satisfied. Let $\mathcal{M}\subseteq\Lambda$ be a set such that $|\mathcal{M}|=4$. Then, it is either the case that
		\begin{equation*}
			\text{every $3$-cycle in $\mathcal{M}$ belongs to Case 1; }
		\end{equation*}
		or
		\begin{equation*}
			\text{every $3$-cycle in $\mathcal{M}$ belongs to Case 2.}
		\end{equation*}
	\end{proposition}
	
Let us first see how, after one has verified the above proposition, one can then easily derive Theorem~\ref{my_main_thm}.
	
	\begin{proof}[Proof of Theorem~\ref{my_main_thm}]
		Pick any two $3$-cycles $p\coloneqq (p_i)_{i=0}^3$ and $q\coloneqq (q_i)_{i=0}^3$ in $\Lambda$. We need to show that $p$ and $q$ belong to the same Case. We have the following scenarios:
		\begin{enumerate}
			\item \textit{The cycles $p$ and $q$ share the same vertices.} In this scenario, $p$ and $q$ are obviously the same $3$-cycle up to a possible reversion, in which case $p$ and $q$ belong to the same Case trivially by definition.
			\item \textit{The cycles $p$ and $q$ share exactly two vertices.} Define $\mathcal{M}$ to be the set of vertices of $p$ and $q$. It is clear that $|\mathcal{M}|=4$ and that $p$ and $q$ are $3$-cycles in $\mathcal{M}$, and so we are done by Proposition~\ref{prop_step_2_lambda_4_elements}.
			\item \textit{The cycles $p$ and $q$ share exactly one vertex.} In this case there is a unique index $i$ such that $q_i$ is a vertex of $p$. Denote by $q_{l_1}$ and $q_{l_2}$ the other two distinct vertices of $q$; and by $p_{m_1}$ and $p_{m_2}$ the other two distinct vertices of $p$. Note how the $3$-cycles $p$ and $(q_i,q_{l_1},p_{m_1},q_i)$ share exactly two vertices; namely, $q_i$ and $p_{m_1}$. By the conclusion of the above scenario, we must then have that the $3$-cycles $p$ and $(q_i,q_{l_1},p_{m_1},q_i)$ belong to the same Case. Furthermore, the $3$-cycles $q$ and $(q_i,q_{l_1},p_{m_1},q_i)$ also share exactly two vertices; namely, $q_i$ and $q_{l_1}$. Hence, for the same reason, it must be the case that the $3$-cycles $q$ and $(q_i,q_{l_1},p_{m_1},q_i)$ belong to the same Case as well. We are done by transitivity.
			\item \textit{The cycles $p$ and $q$ have completely different vertices.} Since the $3$-cycles $p$ and $(q_0,q_1,p_1,q_0)$ share exactly one vertex, it follows, thanks to the conclusion of the above scenario, that the $3$-cycles $p$ and $(q_0,q_1,p_1,q_0)$ belong to the same Case. Additionally, the $3$-cycles $q$ and $(q_0,q_1,p_1,q_0)$ share exactly two vertices; and so by the conclusion of the second scenario, we then have that the $3$-cycles $q$ and $(q_0,q_1,p_1,q_0)$ belong to the same Case. We are done by transitivity.
		\end{enumerate}
		We have exhausted all possible scenarios for the two cycles $p$ and $q$ and have successfully proved in each one of them that the $3$-cycles $p$ and $q$ belong to the same Case, as required. \qed
	\end{proof}
	
	It remains to prove Proposition~\ref{prop_step_2_lambda_4_elements}. We recall that said proposition establishes a relationship between a pair of determinantally equivalent functions, $K$ and $Q$, restricted on sets of cardinality four. As such, to simplify notations, we will assume throughout this section, without loss of generality, that $K$ and $Q$ have domain $\mathcal{M}^2$, where $\mathcal{M}\coloneqq\{1,2,3,4\}$; and we will be consistent with the labelling of the $4$-cycles and $3$-cycles in $\mathcal{M}$ that we had given in (\ref{four_cycle_labellings}) and (\ref{three_cycle_labellings}) from our previous section.  
	
	From the pigeonhole principle, we know that there will always be two $3$-cycles $p^{(i)}$ (out of the four from (\ref{three_cycle_labellings})) belonging to the same Case. Since the action of the symmetric group $S_4$ is transitive on pairs of distinct undirected $3$-cycles on four vertices, we can actually assume, without loss of generality, which specific two $3$-cycles in $\mathcal{M}$ belong to Case 1:
	
\textbf{\underline{Running assumptions for the remainder of Section~\ref{section_proof_of_my_theorem}:}} $p^{(2)}$ and $p^{(4)}$ belong to Case 1.

The remainder of this section is therefore concerned in proving that $p^{(1)}$ and $p^{(3)}$ also belong to Case 1.
	
	We need a couple of lemmas; the first lemma we present below states that if we have knowledge that one of the $3$-cycles $p^{(1)}$ or $p^{(3)}$ belongs to Case 1, then the other necessarily follows suit. 
	
	\begin{lemma}
		\label{prop_step_2.1_lambda_4_elements}
		Let $\mathbb{F}$ be a field, and suppose that functions $K:\mathcal{M}^2\to\mathbb{F}$ and $Q:\mathcal{M}^2\to\mathbb{F}$ are determinantally equivalent and nowhere-zero except possibly on the set $\{(x,x):x\in\mathcal{M}\}$. Suppose that one of the $3$-cycles $p^{(1)}$ or $p^{(3)}$ belongs to Case 1, then so must the other.
	\end{lemma}
	\begin{proof}
	Let us suppose without loss of generality that it is the $3$-cycle $p^{(3)}$ that belongs to Case 1. By Lemma~\ref{graph_theory_lemma_2_from_notes}, equation (\ref{eqn_2by2_1by1}) and the fact that the $3$-cycles $p^{(2)}$, $p^{(3)}$ and $p^{(4)}$ are assumed to belong to Case 1, we have
	\begin{align*}
		K[p^{(1)}]&=\frac{K[p^{(2)}]K'[p^{(3)}]K[p^{(4)}]}{K(4,1)K(1,4)\cdot K(3,4)K(4,3)\cdot K(2,4)K(4,2)} \\
		&=\frac{Q[p^{(2)}]Q'[p^{(3)}]Q[p^{(4)}]}{Q(4,1)Q(1,4)\cdot Q(3,4)Q(4,3)\cdot Q(2,4)Q(4,2)}\\
		&= Q[p^{(1)}].
	\end{align*}	
This proves that the remaining $3$-cycle $p^{(1)}$ also belongs to Case 1. \qed
	\end{proof}
	
	\begin{remark}
		The above lemma does not require condition (\ref{principle_minor_2_by_2_main_eqn}) from our theorem's hypothesis. It does however require $K(x,y)\neq0$ for $x\neq y$, for otherwise we would divide by zero.
	\end{remark}

\begin{remark}
Modulo a relabelling of the vertices of each of the $p^{(i)}$, the above lemma in essence states that if three of the four $3$-cycles, $p^{(1)},\ldots,p^{(4)}$, in $\mathcal{M}$ belong to the same Case, then so must the remaining one.
\end{remark}
	
	The above result implies that if we have knowledge that either of the two cycles $p^{(1)}$ and $p^{(3)}$ belong to Case 1, then the proof of Proposition~\ref{prop_step_2_lambda_4_elements} is complete; so let us suppose henceforth that $p^{(1)}$ and $p^{(3)}$ \textit{both} belong to Case 2. 
	
	\begin{lemma}
		\label{lemma_eq_four_cycle_all_case_identical}
		Let $\mathbb{F}$ be a field, and suppose that functions $K:\mathcal{M}^2\to\mathbb{F}$ and $Q:\mathcal{M}^2\to\mathbb{F}$ are determinantally equivalent and nowhere-zero except possibly on the set $\{(x,x):x\in\mathcal{M}\}$. Then, 
		\begin{equation}
			\label{eq_four_cycle_all_case_equal}
			K[q^{[1]}]+K'[q^{[1]}]=Q[q^{[1]}]+Q'[q^{[1]}].
		\end{equation}		
	\end{lemma}
	\begin{proof}
	It follows from Lemma~\ref{graph_theory_lemma_1_from_notes}, equation (\ref{eqn_2by2_1by1}), as well as our assumption that the $3$-cycles $p^{(1)}$ and $p^{(3)}$ belong to Case 2, that 
	
\begin{equation*}
	K[q^{[1]}]=\frac{K[p^{(1)}]K[p^{(3)}]}{K(1,3)K(3,1)}=\frac{Q'[p^{(1)}]Q'[p^{(3)}]}{Q(1,3)Q(3,1)}=Q'[q^{[1]}],
\end{equation*}
in which case, $K'[q^{[1]}]=Q[q^{[1]}]$. The result now immediately follows. \qed
	\end{proof}

	\begin{lemma}
		\label{lemma_two_diff_sums_equal}
		Let $\mathbb{F}$ be a field, and suppose that functions $K:\mathcal{M}^2\to\mathbb{F}$ and $Q:\mathcal{M}^2\to\mathbb{F}$ are determinantally equivalent. Then,
		\begin{equation}
			\label{leibniz_formula_4_by_4_evaluated2}
			\sum_{i=1}^3 \Big( K[q^{[i]}]+K'[q^{[i]}] \Big) = \sum_{i=1}^3 \Big( Q[q^{[i]}]+Q'[q^{[i]}] \Big).
		\end{equation}
		Moreover, if $K$ and $Q$ are nowhere-zero except possibly on the set $\{(x,x):x\in\mathcal{M}\}$, then 
		\begin{equation}
			\label{leibniz_formula_4_by_4_evaluated2_specific}
			\Big( K[q^{[2]}]+K'[q^{[2]}] \Big) + \Big( K[q^{[3]}]+K'[q^{[3]}] \Big) = \Big( Q[q^{[2]}]+Q'[q^{[2]}] \Big) + \Big( Q[q^{[3]}]+Q'[q^{[3]}] \Big).
		\end{equation}
	\end{lemma}
	\begin{proof}
		Since $|\mathcal{M}|=4$, we can use equation (\ref{equal_determinants_relation}) with $n=4$ in conjunction with the Leibniz formula for determinants to obtain, for $x_1,x_2,x_3,x_4\in\mathcal{M}$ distinct,
		\begin{equation}
			\label{leibniz_formula_4_by_4}
			\sum_{\sigma\in S_4} \text{sgn}(\sigma) \prod_{i=1}^4 K(x_i,x_{\sigma(i)})=\sum_{\sigma\in S_4} \text{sgn}(\sigma) \prod_{i=1}^4 Q(x_i,x_{\sigma(i)}).
		\end{equation}
	It is clear that for a permutation $\sigma\in S_4$ not a $3$- or $4$-cycle, the summands on both sides of (\ref{leibniz_formula_4_by_4}) agree. Therefore, as we had done when we were analyzing equation (\ref{leibniz_equality_for_3_by_3}), we can subtract such terms from both sides of (\ref{leibniz_formula_4_by_4}) and be left with
		\begin{equation*}
			\sum_{i=1}^4 K(p^{(i)}_*,p^{(i)}_*)\Big( K[p^{(i)}]+K'[p^{(i)}] \Big) - \sum_{i=1}^3 \Big( K[q^{[i]}]+K'[q^{[i]}] \Big)
		\end{equation*}
		\begin{equation}
			\label{leibniz_formula_4_by_4_evaluated1}
			=
		\end{equation}
		\begin{equation*}
			\sum_{i=1}^4 Q(p^{(i)}_*,p^{(i)}_*)\Big( Q[p^{(i)}]+Q'[p^{(i)}] \Big) - \sum_{i=1}^3 \Big( Q[q^{[i]}]+Q'[q^{[i]}] \Big),
		\end{equation*}
		where, for $i\in\{1,2,3,4\}$, $p_*^{(i)}\in\mathcal{M}$ denotes the unique vertex not in the cycle $p^{(i)}$. Now, we know from equation (\ref{equal_determinants_relation}) with $n=1$ that for every $i\in\{1,2,3,4\}$, $$K(p^{(i)}_*,p^{(i)}_*)=Q(p^{(i)}_*,p^{(i)}_*).$$
		Since $p^{(i)}$, by Lemma~\ref{prop_about_case1_case2_for_3_cycles}, belongs to either Case 1 or 2, we have
		\begin{equation*}
		\{K[p^{(i)}],K'[p^{(i)}]\}=\{Q[p^{(i)}],Q'[p^{(i)}]\}\quad\forall\enspace i\in\{1,2,3,4\}.
		\end{equation*}
		Therefore, we can subtract the first sum of both the left hand side and right hand side of equation (\ref{leibniz_formula_4_by_4_evaluated1}) and we are done.
		
		The second claim of the lemma follows immediately from Lemma~\ref{lemma_eq_four_cycle_all_case_identical} by subtracting (\ref{eq_four_cycle_all_case_equal}) from the established equation (\ref{leibniz_formula_4_by_4_evaluated2}). \qed  
	\end{proof}
	
	\begin{lemma}
		\label{lemma_two_diff_prods_equal}
		Let $\mathbb{F}$ be a field, and suppose that functions $K:\mathcal{M}^2\to\mathbb{F}$ and $Q:\mathcal{M}^2\to\mathbb{F}$ are determinantally equivalent and nowhere-zero except possibly on the set $\{(x,x):x\in\mathcal{M}\}$. Then,
		\begin{equation*}
			\Big( K[q^{[2]}]+K'[q^{[2]}] \Big) \cdot \Big( K[q^{[3]}]+K'[q^{[3]}] \Big) = \Big( Q[q^{[2]}]+Q'[q^{[2]}] \Big) \cdot \Big( Q[q^{[3]}]+Q'[q^{[3]}] \Big).
		\end{equation*}
	\end{lemma}
	\begin{proof}
	It follows from Lemma~\ref{graph_theory_lemma_3_from_notes}, equation (\ref{eqn_2by2_1by1}), our assumption that the $3$-cycles $p^{(2)}$ and $p^{(4)}$ belong to Case 1, as well as our assumption that the $3$-cycles $p^{(1)}$ and $p^{(3)}$ belong to Case 2, that
\begin{equation}
	\label{leqq1}
	K[q^{[2]}]K[q^{[3]}]=K(1,3)K(3,1)\cdot K[p^{(2)}]K'[p^{(4)}]=Q(1,3)Q(3,1)\cdot Q[p^{(2)}]Q'[p^{(4)}]=Q[q^{[2]}]Q[q^{[3]}],
\end{equation}	
 and
 \begin{equation}
 	\label{leqq2}
 	K[q^{[2]}]K'[q^{[3]}]=K(2,4)K(4,2)\cdot K[p^{(1)}]K'[p^{(3)}]=Q(2,4)Q(4,2)\cdot Q'[p^{(1)}]Q[p^{(3)}]=Q'[q^{[2]}]Q[q^{[3]}];
 \end{equation}
in which case,
\begin{equation}
	\label{leqq3}
K'[q^{[2]}]K[q^{[3]}]=Q[q^{[2]}]Q'[q^{[3]}]
\end{equation}
and
\begin{equation}
	\label{leqq4}
K'[q^{[2]}]K'[q^{[3]}]=Q'[q^{[2]}]Q'[q^{[3]}].
\end{equation}
The result now follows immediately. \qed
	\end{proof}

	\begin{corollary}
		\label{cor_final}
		Let $\mathbb{F}$ be a field, and suppose that functions $K:\mathcal{M}^2\to\mathbb{F}$ and $Q:\mathcal{M}^2\to\mathbb{F}$ are determinantally equivalent and nowhere-zero except possibly on the set $\{(x,x):x\in\mathcal{M}\}$. Then, one of the following two statements holds:
		\begin{enumerate}[(i)]
			\item $\begin{aligned}[t]
				K &[q^{[2]}]=Q[q^{[2]}]\text{ and } K'[q^{[2]}]=Q'[q^{[2]}]\quad\text{ or }\quad K[q^{[2]}]=Q'[q^{[2]}]\text{ and } K'[q^{[2]}]=Q[q^{[2]}], \\
				&\qquad\qquad\qquad\qquad\qquad\qquad\qquad\qquad\qquad\text{and} \\
				K &[q^{[3]}]=Q[q^{[3]}]\text{ and } K'[q^{[3]}]=Q'[q^{[3]}]\quad\text{ or }\quad K[q^{[3]}]=Q'[q^{[3]}]\text{ and } K'[q^{[3]}]=Q[q^{[3]}]; \\
			\end{aligned}$
			\item $\begin{aligned}[t]
				K[q^{[2]}]+K'[q^{[2]}]=Q[q^{[3]}]+Q'[q^{[3]}]\text{ and }K[q^{[3]}]+K'[q^{[3]}]=Q[q^{[2]}]+Q'[q^{[2]}].
			\end{aligned}$
		\end{enumerate}
	\end{corollary}
	\begin{proof}
		By Lemma~\ref{lemma_two_diff_sums_equal} (more specifically, equation (\ref{leibniz_formula_4_by_4_evaluated2_specific}) therein) and  Lemma~\ref{lemma_two_diff_prods_equal}, we have, by a simple application of Lemma~\ref{key_trick_lemma}, that either statement (ii) holds; or the equations
		\begin{equation*}
			K[q^{[2]}]+K'[q^{[2]}]=Q[q^{[2]}]+Q'[q^{[2]}]\text{ and }K[q^{[3]}]+K'[q^{[3]}]=Q[q^{[3]}]+Q'[q^{[3]}]
		\end{equation*}
	hold. We note that, thanks to equation (\ref{eqn_2by2_1by1}), we also have the equations $$K[q^{[2]}]\cdot K'[q^{[2]}]=Q[q^{[2]}]\cdot Q'[q^{[2]}]\text{ and   }K[q^{[3]}]\cdot K'[q^{[3]}]=Q[q^{[3]}]\cdot Q'[q^{[3]}]$$ in our possession. An application of Lemma~\ref{key_trick_lemma} then brings up statement (i). \qed
	\end{proof}
	
	With this corollary in our disposal, we are now in the position to place the last piece of the puzzle on our board, that is, to prove that, in fact, the $3$-cycles $p^{(1)}$ and $p^{(3)}$ actually belong to Case 1:
	
		First suppose there is some index $i\in\{1,2,3,4\}$ such that $K[p^{(i)}]=K'[p^{(i)}]$. Then, by Lemma~\ref{prop_about_case1_case2_for_3_cycles}, $$Q[p^{(i)}]=K[p^{(i)}]=K'[p^{(i)}]=Q'[p^{(i)}],$$ that is, $p^{(i)}$ belongs to Case 1 and 2 simultaneously. Consequently, three of $p^{(1)},p^{(2)},p^{(3)},p^{(4)}$ belong to the same Case; and so then we are done by Lemma~\ref{prop_step_2.1_lambda_4_elements} (cf., the second remark after the lemma). So let us suppose henceforth that
		\begin{equation}
			\label{non_equality_of_K_and_K_prime}
			K[p^{(i)}]\neq K'[p^{(i)}]\quad\forall i\in\{1,2,3,4\}.
		\end{equation}
		In other words, we are supposing that each of the $3$-cycles $p^{(i)}$ belong to exactly one Case. Specifically, we are supposing that $p^{(2)}$ and $p^{(4)}$ \textit{only} belong to Case 1, and that $p^{(1)}$ and $p^{(3)}$ \textit{only} belong to Case 2. We seek for a contradiction. We have two scenarios to examine corresponding to the cases in Corollary~\ref{cor_final}; let us start with case (i).
		
		 It follows from Lemma~\ref{graph_theory_lemma_1_from_notes} (through a simple relabelling of vertices) that for every function $h:\mathcal{M}^2\to\mathbb{F}$ that is nowhere-zero, except perhaps on the set $\{(x,x):x\in\mathcal{M}\}$, that
		\begin{equation}
			\label{start_eq1}
			h[q^{[2]}]=\frac{h[p^{(1)}]h'[p^{(4)}]}{h(2,3)h(3,2)}=\frac{h'[p^{(3)}]h[p^{(2)}]}{h(1,4)h(4,1)}.
		\end{equation}
		Suppose first that the instance $K[q^{[2]}]=Q[q^{[2]}]$ from (i) of Corollary~\ref{cor_final} holds. Then, (\ref{start_eq1}) in conjunction with equation (\ref{eqn_2by2_1by1}) yields
		\begin{equation*}
			\frac{K[p^{(1)}]K'[p^{(4)}]}{K(2,3)K(3,2)}=K[q^{[2]}]=Q[q^{[2]}]=\frac{Q[p^{(1)}]Q'[p^{(4)}]}{Q(2,3)Q(3,2)}=\frac{Q[p^{(1)}]Q'[p^{(4)}]}{K(2,3)K(3,2)}.
		\end{equation*}
		Thus, $$K[p^{(1)}]K'[p^{(4)}]=Q[p^{(1)}]Q'[p^{(4)}].$$
		Since we have assumed $p^{(4)}$ to be a $3$-cycle belonging to Case 1, the above equation implies that $p^{(1)}$ also belongs to Case 1; contradicting our earlier assumption that $p^{(1)}$ \textit{only} belongs to Case 2. 
		
		Let us now suppose that the instance $K[q^{[2]}]=Q'[q^{[2]}]$ from (i) of Corollary~\ref{cor_final} holds. Then, (\ref{start_eq1}) in conjunction with equation (\ref{eqn_2by2_1by1}) yields
	\begin{equation*}
		\frac{K[p^{(1)}]K'[p^{(4)}]}{K(2,3)K(3,2)}=K[q^{[2]}]=Q'[q^{[2]}]=\frac{Q'[p^{(1)}]Q[p^{(4)}]}{Q(2,3)Q(3,2)}=\frac{Q'[p^{(1)}]Q[p^{(4)}]}{K(2,3)K(3,2)}.
	\end{equation*}	
Thus,
\begin{equation*}
K[p^{(1)}]K'[p^{(4)}]=Q'[p^{(1)}]Q[p^{(4)}].
\end{equation*}
Since we have assumed $p^{(1)}$ to be a $3$-cycle belonging to Case 2, the above equation implies that $p^{(4)}$ also belongs to Case 2; contradicting our earlier assumption that $p^{(4)}$ \textit{only} belongs to Case 1. 		
		
We conclude that (i) of Corollary~\ref{cor_final} cannot possibly hold; so let us now assume that (ii) of Corollary~\ref{cor_final} holds. In the proof of Lemma~\ref{lemma_two_diff_prods_equal} we had established equations (\ref{leqq1}) -- (\ref{leqq4}). By dividing (\ref{leqq1}) by (\ref{leqq3}), we get
		\begin{equation}
			\label{m2}
			K[q^{[2]}]Q'[q^{[3]}]=K'[q^{[2]}]Q[q^{[3]}].
		\end{equation}
		Now, by our usual application of Lemma~\ref{key_trick_lemma}, the first equation from (ii) of Corollary~\ref{cor_final} together with equation (\ref{m2}) brings about the following two possibilities:
		\begin{enumerate}[(I)]
			\item $\begin{aligned}[t]
				\label{n1}
				K &[q^{[2]}]+K'[q^{[2]}]=0\text{ and }Q[q^{[3]}]+Q'[q^{[3]}]=0; \\
			\end{aligned}$
			\item $\begin{aligned}[t]
				\label{n2}
				K &[q^{[2]}]=Q[q^{[3]}]\text{ and }K'[q^{[2]}]=Q'[q^{[3]}].
			\end{aligned}$
		\end{enumerate}
	Let us first assume that (I) holds. By dividing (\ref{leqq2}) by (\ref{leqq1}), we get
	\begin{align*}
		\label{mj1}
		\frac{Q'[q^{[2]}]}{Q[q^{[2]}]}=\frac{K'[q^{[3]}]}{K[q^{[3]}]}&=-\frac{K'[q^{[3]}]}{K[q^{[3]}]}\frac{K[q^{[2]}]}{K'[q^{[2]}]} &&\text{(by assumption (I), $K[q^{[2]}]=-K'[q^{[2]}]$)} \\
		&=-\frac{K(2,4)K(4,2)\cdot K[p^{(1)}]K'[p^{(3)}]}{K(2,4)K(4,2)\cdot K'[p^{(1)}]K[p^{(3)}]} &&\text{(by Lemma~\ref{graph_theory_lemma_3_from_notes}, equation (\ref{graph_theory_lemma_3_from_notes_eq2}))} \\
		&= -1,
	\end{align*}
where the last line follows from our assumption that the $3$-cycles $p^{(1)}$ and $p^{(3)}$ belong to Case 2. Therefore, 
$$K[q^{[2]}]+K'[q^{[2]}]=0=Q[q^{[2]}]+Q'[q^{[2]}].$$	
By equation (\ref{eqn_2by2_1by1}), we also have $K[q^{[2]}]\cdot K'[q^{[2]}]=Q[q^{[2]}]\cdot Q'[q^{[2]}]$; and so, by the usual application of Lemma~\ref{key_trick_lemma}, we have
\begin{equation*}
	K[q^{[2]}]=Q[q^{[2]}]\text{ and } K'[q^{[2]}]=Q'[q^{[2]}]\quad\text{ or }\quad K[q^{[2]}]=Q'[q^{[2]}]\text{ and } K'[q^{[2]}]=Q[q^{[2]}].
\end{equation*}
But this is precisely the first part of (i) of Corollary~\ref{cor_final}, which we had investigated earlier and had obtained a contradiction. 	
	
The remaining scenario to check is when (II) holds; so let us now assume that (II) holds. Thanks to (\ref{start_eq1}), equation (\ref{eqn_2by2_1by1}), and the fact that the $3$-cycles $p^{(1)}$ and $p^{(4)}$ belong to Cases 2 and 1, respectively, we have
		\begin{equation*}
			K[q^{[2]}]=\frac{K[p^{(1)}]K'[p^{(4)}]}{K(2,3)K(3,2)}=\frac{Q'[p^{(1)}]Q'[p^{(4)}]}{Q(2,3)Q(3,2)}.
		\end{equation*}
	This equation in conjunction with the assumed $K[q^{[2]}]=Q[q^{[3]}]$ of (II) yields
	\begin{equation*}
	\cancel{Q(1,3)}\cancel{Q(3,2)}\cancel{Q(2,4)}Q(4,1)=\frac{\cancel{Q(1,3)}\cancel{Q(3,2)}Q(2,1)\cdot \cancel{Q(2,4)}Q(4,3)\cancel{Q(3,2)}}{Q(2,3)\cancel{Q(3,2)}},
	\end{equation*}
	which simplifies to
	\begin{equation*}
	Q(2,3)Q(4,1)=Q(2,1)Q(4,3),
	\end{equation*}
that is,
	\begin{equation*}
		\begin{vmatrix}
			Q(2,1) & Q(2,3) \\
			Q(4,1) & Q(4,3)
		\end{vmatrix}=0;	
	\end{equation*}
	a contradiction to our theorem's hypothesis. Since Proposition~\ref{prop_step_2_lambda_4_elements} is now proven, we can formally conclude Theorem~\ref{my_main_thm}.

	\section*{Acknowledgements}
	The author is supported by the Warwick Mathematics Institute Centre for Doctoral Training, and gratefully acknowledges funding from the University of Warwick. The author would like to thank the anonymous referee for their valuable suggestions that significantly improved the clarity and exposition of the paper.

	\bibliographystyle{plain}
	\bibliography{det_equiv_nonzero_fns_v2_bibliography.bib}
\end{document}